\documentclass[11pt, a4paper]{article}
\title{\bf Sur le th\'eor\`eme  de Brauer--Siegel g\'en\'eralis\'e}
\author{{\Large Richard \textsc{Griffon}} \and {\Large Philippe \textsc{Lebacque}} \and {\Large Ga\"el \textsc{R\'emond}}}
\date{}

\usepackage{geometry}
\usepackage[frenchb]{babel}
\usepackage[utf8]{inputenc}
\usepackage{lmodern}
\usepackage{amssymb,amsmath,amscd, amsthm}  
\usepackage{mathrsfs}
\usepackage{tikz-cd}
\usepackage{enumerate}
\usepackage{graphicx}
\usepackage{url}
\usepackage{multicol}
\newtheoremstyle{namedtheo}%
  {}{}%
  {\itshape}
  {}{\bfseries}
  {.}{ }
  {\thmname{#1}\thmnumber{ #2}\thmnote{\,(#3)}}
	\newtheorem{theo}{Th\'eor\`eme}[section]
	\newtheorem{introtheo}{Th\'eor\`eme}
	\newtheorem{coro}[theo]{Corollaire}
	\newtheorem{lemm}[theo]{Lemme}
	\newtheorem{prop}[theo]{Proposition}
	\newtheorem{introprop}[introtheo]{Proposition}
{\theoremstyle{definition}
	\newtheorem{defi}[theo]{D\'efinition}
	\newtheorem{exple}[theo]{Exemple}
	\newtheorem{rema}[theo]{Remarque}}
{\theoremstyle{namedtheo}
	\newtheorem*{theono}{Th\'eor\`eme}
	\newtheorem*{conjno}{Conjecture}}

\usepackage{titlesec}
\titleformat{\subsection}[runin]%
        {\bfseries}%
        {\thesubsection.}%
        {0.2em}{}[.\hspace{0.4em}-- ]        
 \titleformat{\section}%
        {\center\Large\bfseries}%
        {\thesection.}%
        {0.2em}{}[]%
\usepackage[colorlinks=true, citecolor= teal!70, linkcolor= blue!70, urlcolor=black!60, pagebackref]{hyperref}
\newcommand{\ds}{\displaystyle}
\newcommand{\solv}{^{\mathrm{solv}}}
\newcommand{\N}{\mathbb{N}}
\newcommand{\Z}{\mathbb{Z}}
\newcommand{\Q}{\mathbb{Q}}
\newcommand{\R}{\mathbb{R}}
\newcommand{\C}{\mathbb{C}}

\newcommand{\Lcal}{\mathscr{L}} 
\newcommand{\Kcal}{\mathscr{K}}

\newcommand{\Hcacal}{\mathcal{H}}

\newcommand{\BS}{\mathcal{BS}}

\DeclareMathOperator{\Gal}{Gal}
\renewcommand{\ker}{\mathrm{Ker}\,}
\newcommand{\im}{\mathrm{Im}\,}
\newcommand{\comp}{\xi}
\newcommand{\clgal}{\eta}
\newcommand{\itref}[1]{{\it (\ref{#1})}}
\newcommand{\refcond}[1]{\hyperref[#1]{{\it(C\ref*{#1})}}}

\begin{document}
\maketitle 
\vspace{-2em}\begin{flushright}
\textit{\`A la m\'emoire de notre ami Alexey Zykin,\\ dont les id\'ees nous inspirent toujours,\\ et de sa femme Tatyana}
\end{flushright}\medskip

\noindent\hfill\rule{7cm}{0.5pt}\hfill\phantom{.}
\medskip

\noindent \textbf{R\'esum\'e --} 
Nous \'etendons le th\'eor\`eme de Brauer--Siegel \`a de nouvelles familles de corps de nombres,  comme par exemple celles form\'ees de corps contenus dans la cl\^oture r\'esoluble d'un corps de nombres fix\'e. Nous d\'efinissons pour cela une notion de complexit\'e galoisienne d'une extension de corps de nombres et montrons que toute famille asymptotiquement exacte dans laquelle la complexit\'e galoisienne ne cro\^it pas trop vite v\'erifie le th\'eor\`eme de Brauer--Siegel g\'en\'eralis\'e conjectur\'e par Tsfasman--Vl\u{a}du\c{t}. Ce crit\`ere unifie et \'etend divers r\'esultats ant\'erieurs. L'essentiel de la d\'emonstration consiste \`a \'etablir un nouveau principe de descente de z\'eros pour les fonctions z\^eta des corps de nombres qui raffine celui de Stark en faisant intervenir notre complexit\'e galoisienne.
\medskip

\noindent \textbf{Abstract --} 
We extend the Brauer--Siegel theorem to new families of number fields, both in the classical setting of asymptotically bad families and in the more general framework due to Tsfasman and Vl\u{a}du\c{t} of asymptotically exact families. We introduce a notion of Galois complexity for extensions of number fields, and show that the generalized Brauer--Siegel theorem, as conjectured by Tsfasman and Vl\u{a}du\c{t}, holds for families in which the complexity does not grow too fast. This allows to unify and extend all previously known results. The crucial step in our work is the proof of a new version -- stated in terms of our Galois complexity -- of a fundamental principle due to Stark descending exceptional zeroes of zeta functions down to quadratic number fields. Among the hitherto unknown cases we are able to treat are the families of number fields contained in the solvable Galois closure of a given number field.
\medskip

\noindent{\it Mots-Cl\'es:} 
Th\'eorie asymptotique des corps de nombres, Th\'eor\`eme de Brauer--Siegel.

\noindent{\it 2020 Math. Subj. Classification:}  
11R42, 
11R47, 
11R99. 

\noindent\hfill\rule{7cm}{0.5pt}\hfill\phantom{.}
\medskip

\section*{Introduction}
\pdfbookmark[0]{Introduction}{Introduction} 
\setcounter{section}{0}

 Le th\'eor\`eme de Brauer--Siegel d\'ecrit le comportement asymptotique du produit du nombre de  classes~$h(K)$  d'un corps de nombres $K$  par son r\'egulateur des unit\'es $R(K)$ en termes de son discriminant~$\Delta(K)$.  C'est un r\'esultat central de la th\'eorie des nombres du XX\ieme{} si\`ecle, qui a trouv\'e de nombreuses applications \`a  des domaines vari\'es des math\'ematiques (par exemple  aux empilements de sph\`eres \cite{RT},  aux exposants de groupes de classes \cite{HM18},  en g\'eom\'etrie anab\'elienne \cite{Iv14},  pour la conjecture d'Andr\'e--Oort pour $\mathcal{A}_g$ \cite{Tsim18})  et de l'informatique (th\'eorie des graphes \cite{ST},  cryptosyst\`emes \`a cl\'e publique \cite{HM20}).  R\'ecemment, le th\'eor\`eme de Brauer--Siegel a \'egalement donn\'e lieu \`a des \'enonc\'es en dimension sup\'erieure: ainsi, on pourra consulter \cite{HP16} et  \cite{GriffonPHD16, GriffonLegendre18, GriffonAS18} pour un analogue dans certaines suites de courbes elliptiques.    Mentionnons enfin que, faisant suite \`a la preuve dans \cite{GriffonFermat18} d'un th\'eor\`eme de Brauer--Siegel   pour la suite des surfaces de Fermat sur un corps fini,   Hindry \cite{HindryBSgeneral19} a propos\'e une vaste g\'en\'eralisation conjecturale pour des  suites de vari\'et\'es sur les corps finis.  
 \medskip
 
Le th\'eor\`eme de Brauer--Siegel classique peut s'\'enoncer de la mani\`ere suivante  (voir \cite{Siegel35}, \cite{Brauer47}, \cite{Brauer50} ou bien les chapitres XIII et XVI de \cite{LangANT}).
		\begin{theono}[Brauer--Siegel]\label{itheo:BS.classique}
		 Soit $(K_i)_{i\in\N}$ une suite de corps de nombres telle que
		 \begin{equation}\label{ieq:asymp.bad}\tag{M}
	 	 	\lim_{i\to\infty} \frac{{\log|\Delta(K_i)|}}{[K_i:\Q]} = +\infty.
		 \end{equation} 
		Alors on a 
			\[\lim_{i\to\infty}\frac{\log\big(h(K_i)\,  R(K_i)\big)}{\log\sqrt{|\Delta(K_{i})}|} = 1\]
		si l'une des conditions suivantes est remplie : 
		\begin{enumerate}[(C1)]
		\setcounter{enumi}{-1}
			\item\label{c:GRH} l'hypoth\`ese de Riemann g\'en\'eralis\'ee est vraie,
			\item\label{c:degborne} la suite $([K_i:\Q])_{i\in\N}$ est born\'ee,
			\item\label{c:galois} pour tout $i\in\N$ l'extension $K_i/\Q$ est galoisienne.
		\end{enumerate}
		\end{theono} 

Plus r\'ecemment, Zykin, Dixit et Wong ont donn\'e des renforcements de cet \'enonc\'e en rempla\c{c}ant la condition \refcond{c:galois} par d'autres hypoth\`eses de nature galoisienne. Pour les formuler, rappelons une d\'efinition. Nous dirons qu'une extension $L/K$ de corps de nombres est \emph{galoisienne par pas} s'il existe une suite $K=L_0 \subset L_1 \subset \dots\subset L_s=L$  de sous-extensions telle que, pour tout $i\in\{1,...,s\}$, l'extension $L_{i}/L_{i-1}$ est galoisienne. Avec ce vocabulaire, les articles \cite{Zykin}, \cite{Dixit} et \cite{Wong22} \'etablissent respectivement le th\'eor\`eme sous les conditions:
{\it \begin{enumerate}[\it (C1)]
\setcounter{enumi}{2}
	\item\label{c:galoisparpas} pour tout $i\in\N$ l'extension $K_i/\Q$ est galoisienne par pas, 
	\item\label{c:clotgalres} pour tout $i\in\N$	 la cl\^oture galoisienne de l'extension $K_i/\Q$ est r\'esoluble,
	\item\label{c:wong} pour tout $i\in\N$ il existe un corps $\Q\subset K'_i\subset K_i$ tel que la cl\^oture galoisienne de $K_i/K'_i$ est r\'esoluble et $K'_i/\Q$ est galoisienne par pas.
\end{enumerate} }

Notre but consiste \`a affaiblir encore l'hypoth\`ese du th\'eor\`eme pour traiter des suites de corps beaucoup plus g\'en\'erales qui, en particulier, engloberont naturellement celles qui v\'erifient l'une des conditions \refcond{c:degborne} \`a \refcond{c:wong}. Nous rempla\c{c}ons de plus \eqref{ieq:asymp.bad} par une hypoth\`ese moins restrictive en suivant Tsfasman et Vl\u{a}du\c{t} (voir \cite{TV}). En effet, ils ont remarqu\'e que, pour certaines suites $\Kcal=(K_i)_{i\in\N}$ ne v\'erifiant pas \eqref{ieq:asymp.bad}, la limite de $\log\big(h(K_i)\,R(K_i)\big)/\log\sqrt{|\Delta(K_i)|}$ peut exister et \^etre diff\'erente de $1$. Nous rappellerons dans la partie suivante comment ils d\'efinissent une notion de famille asymptotiquement exacte $\Kcal$ et lui associent un nombre r\'eel $\beta(\Kcal)\in[2/5, 6/5]$. Les suites $\Kcal$ v\'erifiant \eqref{ieq:asymp.bad} sont asymptotiquement exactes et satisfont $\beta(\Kcal)=1$. D'autres exemples sont obtenus en consid\'erant les \emph{tours}, c'est-\`a-dire les suites avec $K_i\subset K_{i+1}$ pour tout $i\in\N$. Dans ce langage, les travaux de Tsfasman et Vl\u{a}du\c{t} sugg\`erent la conjecture suivante.

			\begin{conjno}[Tsfasman--Vl\u{a}du\c{t}] 
			 Si $\Kcal=(K_i)_{i\in\N}$ est une famille asymptotiquement exacte de corps de nombres, nous avons
			\begin{equation*}
			\lim_{i\to\infty}\frac{\log\big(h(K_i) \, R(K_i)\big)}{\log\sqrt{|\Delta(K_{i})|}} 
			= \beta(\Kcal). 
			 \end{equation*}
			\end{conjno}

Nous dirons qu'une famille asymptotiquement exacte de corps de nombres \emph{v\'erifie le th\'eor\`eme de Brauer--Siegel g\'en\'eralis\'e} si la conjecture de Tsfasman--Vl\u{a}du\c{t} est vraie pour cette famille. Comme le th\'eor\`eme classique, cette conjecture est vraie sous l'hypoth\`ese \refcond{c:GRH}   d'apr\`es Tsfasman--Vl\u{a}du\c{t} eux-m\^emes et sous \refcond{c:galoisparpas} d'apr\`es Lebacque \cite{Lebacque}. On la conna\^it \'egalement sous \refcond{c:clotgalres} ou \refcond{c:wong} si l'on se restreint au cas des tours (voir \cite{Dixit} et \cite{Wong22} respectivement).

Ici encore, notre th\'eor\`eme principal unifie et g\'en\'eralise tous ces travaux. Pour l'\'enoncer, introduisons la d\'efinition suivante: si $L/K$ est une extension de corps de nombres, sa complexit\'e galoisienne $\comp_{L/K}$ est l'entier
	 \[\comp_{L/K} = \min\left\{ \max_{1\leq i \leq s} [L_i:L_{i-1}]_{\Gal}, \ K=L_0 \subset L_1 \subset \dots \subset L_s=L\right\}, \]
o\`u $[L_i:L_{i-1}]_{\Gal}$ d\'esigne le degr\'e de la cl\^oture galoisienne de $L_i/L_{i-1}$  et le minimum est pris sur toutes les cha\^ines finies de sous-extensions de $L/K$.
		\begin{introtheo}\label{itheo:principal}
		Le th\'eor\`eme de Brauer--Siegel g\'en\'eralis\'e est vrai pour toute famille asymptotiquement exacte $(K_i)_{i\in\N}$ de corps de nombres telle que 
		\begin{equation*}
			\lim_{i\to\infty} \frac{\log|\Delta(K_i)|}{\log\comp_{K_i/\Q}} = +\infty.
		\end{equation*}
		\end{introtheo}

Nous constatons que, pour tous corps de nombres $K\subset L \subset M$, nous avons 
	\[ \comp_{M/K}\leq \max\big\{\comp_{M/L}, \comp_{L/K}\big\}, \qquad  1\leq \comp_{L/K}\leq [L:K]!\] 
et, si $L/K$ est galoisienne par pas, $\comp_{L/K}\leq [L:K]$. Nous montrerons aussi (voir lemme \ref{lemm:compl.gal.clotgalres}) \linebreak $\log\comp_{L/K}\leq 3(\log[L:K])^2$ si la cl\^oture galoisienne de $L/K$ est r\'esoluble. Ces in\'egalit\'es permettent de v\'erifier rapidement que chacune des conditions \refcond{c:degborne} \`a \refcond{c:wong}  entra\^ine l'hypoth\`ese du th\'eor\`eme \ref{itheo:principal}, qui englobe donc bien les r\'esultats cit\'es ci-dessus. Les familles obtenues sont bien plus g\'en\'erales : \`a titre d'exemple, les in\'egalit\'es pr\'ec\'edentes montrent aussi que le th\'eor\`eme \ref{itheo:principal} couvre les familles asymptotiquement exactes $(K_i)_{i\in \N}$ v\'erifiant
{\it 
\begin{enumerate}[\it (C1)]
\setcounter{enumi}{5}
	\item\label{c:admissible} il existe un entier $\mu$ tel que, pour tout $i\in\N$, il existe une suite de corps 
		\[\Q= L_{i, 0}\subset L_{i, 1} \subset \dots \subset L_{i, s_i} = K_i\] de sorte que, pour $1\leq j\leq s_i$, l'extension $L_{i,j}/L_{i, j-1}$ est soit galoisienne, soit de cl\^oture galoisienne r\'esoluble, soit de degr\'e au plus $\mu$.
\end{enumerate}}
\medskip

Toutes les approches du th\'eor\`eme de Brauer--Siegel se basent sur la formule des classes de Dirichlet pour traduire le probl\`eme en termes de fonctions z\^eta. Notons $\zeta_K$ la fonction z\^eta de Dedekind d'un corps de nombres $K$ et $\zeta_K^\ast(1)$ son r\'esidu en $1$. Dans les \'enonc\'es ci-dessus, il s'agit alors d'encadrer $\zeta_{K_i}^\ast(1)$ et toute la difficult\'e r\'eside dans la minoration en raison de la pr\'esence \'eventuelle de z\'eros r\'eels de $\zeta_{K_i}$ proches de $1$. Sous la condition \refcond{c:GRH} ces z\'eros n'existent pas; sans cette information, il faut t\^acher de montrer qu'ils ne peuvent pas tendre trop vite vers $1$ lorsque $i$ tend vers l'infini. Dans notre situation, nous verrons en reprenant l'approche de Lebacque \cite{Lebacque} bas\'ee sur le th\'eor\`eme de Mertens que le th\'eor\`eme \ref{itheo:principal} d\'ecoule assez rapidement du r\'esultat originel de Siegel \cite{Siegel35} pour les corps quadratiques et du principe de descente de z\'eros suivant.
		\begin{introtheo}\label{itheo:desc.gael}
		Soit $L/K$ une extension de corps de nombres. Si la fonction z\^eta $\zeta_L $ admet un z\'ero r\'eel $\rho$ v\'erifiant
		\[\rho \in \left[1 - (32g(L)\,\comp_{L/K})^{-1} , 1\right[,\]
		alors il existe une sous-extension $F$ de $L/K$ avec $[F:K]\leq 2$ et $\zeta_F(\rho)=0$.
		\end{introtheo}

Cet \'enonc\'e g\'en\'eralise un r\'esultat fondamental de Stark \cite{Stark} et une variante due \`a Murty \cite{Murty01} qui \'etaient utilis\'es dans les travaux ant\'erieurs de Tsfasman--Vl\u{a}du\c{t}, Zykin, Lebacque, Dixit et Wong. Notons que nous nous int\'eressons uniquement ici au th\'eor\`eme de Brauer--Siegel mais le th\'eor\`eme \ref{itheo:desc.gael} permet aussi d'am\'eliorer les autres applications du r\'esultat de Stark: par exemples les th\'eor\`emes 1 et 2 de \cite{Stark} sont valables en rempla\c{c}ant $g(n)$ par $\comp_{k/\Q}$ (voir le paragraphe \ref{ss:commentaires} ci-dessous). 
\medskip

Nous mentionnons pour terminer quelques cons\'equences du th\'eor\`eme \ref{itheo:principal}.
		\begin{introtheo}\label{itheo:ext.infinies}
		Soient $(K_i)_{i\in\N}$ et $(L_i)_{i\in\N}$ une famille et une tour de corps de nombres.  
		Si 
		\[\lim_{i\to\infty} \frac{\log|\Delta(K_i)|}{\log\comp_{K_i/\Q}} = +\infty \qquad \text{ et }\qquad \bigcup_{i\in\N} K_i \subset \bigcup_{i\in\N}L_i,\]
		alors la tour $(L_i)_{i\in\N}$ v\'erifie le th\'eor\`eme de Brauer--Siegel g\'en\'eralis\'e.
		\end{introtheo}
Ce crit\`ere assez souple s'applique \`a de nombreuses tours.
		\begin{introtheo}\label{itheo:solvable}
		Soit $K$ un corps de nombres.  
		Toute famille asymptotiquement exacte de corps de nombres contenus dans sa cl\^oture r\'esoluble $K\solv$ v\'erifie le th\'eor\`eme de Brauer--Siegel g\'en\'eralis\'e.
		\end{introtheo}

Puisqu'un $p$-groupe est r\'esoluble, nous en d\'eduisons en particulier:
		\begin{introprop} \label{prop:Sptour}
		Soient $K$ un corps de nombres, $p$ un nombre premier et $S$ un ensemble 
		de places de $K$. 
		Soit $K_S(p)$ la pro-$p$-extension maximale de $K$ non ramifi\'ee en dehors de $S$.
		Toute famille asymptotiquement exacte  de corps de nombres contenus dans $K_S(p)$ 
		v\'erifie le th\'eor\`eme de Brauer--Siegel g\'en\'eralis\'e. 
		\end{introprop}

L'int\'er\^et de citer cet exemple tient à ce que ces corps $K_S(p)$ fournissent essentiellement  la seule fa\c{c}on connue de construire explicitement (lorsque $S$ est fini et $[K_S(p):K]$ infini) des tours de corps de nombres ne satisfaisant pas \eqref{ieq:asymp.bad} et donc d'exhiber des situations o\`u $\beta(\Kcal)\neq 1$. Nous rappellerons en partie \ref{sec:exemples} quelques unes de ces constructions dues \`a Hajir, Maire, Tsfasman, Vl\u{a}du\c{t} et Zykin.

\numberwithin{equation}{section}    
\section{D\'efinitions et notations}
\label{section:notations}

\subsection{Genre d'un corps de nombres} 
Si $K$ est un corps de nombres, son \emph{genre}, not\'e $g(K)$, est d\'efini par $g({K})=\log{\sqrt{|\Delta(K)|}}$. Cette d\'efinition est motiv\'ee par l'analogie entre l'arithm\'etique des corps de nombres et celle des corps de fonctions des courbes sur les corps finis. Il est classique (th\'eor\`eme de Minkowski, voir proposition 2.17 au chapitre III de \cite{Neukirch}) que le seul corps de nombres de genre~$0$ est $\Q$. Si $K\neq\Q$, l'in\'egalit\'e de Minkowski s'\'ecrit 
\begin{equation}\label{ineq:minkowski}
[K:\Q]\leq \frac{2\,g(K)}{g(\Q(\sqrt{-3}))} = \frac{4}{\log 3}\, g(K) 
\end{equation}
(l'inégalit\'e usuelle \cite[III, proposition 2.14]{Neukirch} entra\^ine ceci si $[K:\Q]\geq 5$ tandis que, pour les degr\'es $2, 3$ et $4$, les corps de discriminant minimal sont connus \cite[\S2]{Poitou}).

Si $L/K$ est une extension finie de corps de nombres, rappelons que l'on a $g(L)\geq [L:K]\, g(K)$.  Il s'agit d'une version faible de la formule de Riemann--Hurwitz qui exprime  plus g\'en\'eralement  la diff\'erence  $g(L)- [L:K]\, g(K)$ en termes de la ramification de l'extension $L/K$ (voir proposition 3.13 de \cite[III]{Neukirch} pour un \'enonc\'e pr\'ecis dans des notations l\'eg\`erement différentes).  En particulier, nous d\'eduisons de cette formule l'\'egalit\'e $g(L)= [L:K]\, g(K)$ lorsque l'extension $L/K$ est partout non ramifi\'ee.

\subsection{Invariants de Tsfasman--Vl\u{a}du\c{t}}\label{ss:invariants}
Soient $\mathcal{Q}_0\subset\Z$ l'ensemble des puissances des nombres premiers et $\mathcal{Q}$ l'ensemble form\'e de $\mathcal{Q}_0$ ainsi que des deux symboles $\R$ et $\C$.

Soit $K$ un corps de nombres. Pour toute puissance $q\in\mathcal Q_0$  nous notons $\Phi_q(K)$ le nombre de places de~$K$ de norme $q$, $\Phi_\mathbb{R}(K)$  le nombre de places r\'eelles et $\Phi_\mathbb{C}(K)$ le nombre de places complexes de $K$. \\

Consid\'erons \`a pr\'esent une suite $\Kcal=(K_i)_{i\in \N}$ de corps de nombres. Nous dirons que la suite $\Kcal$ est \emph{une famille} si $g(K_i)$ tend vers $+\infty$ lorsque $i\to\infty$ et si $K_i\neq \Q$ pour tout $i$.  Pour tout $q\in\mathcal{Q}$ nous d\'efinissons, si cela a un sens, la limite suivante:
\[ \phi_{q}(\Kcal):=\lim_{i\to\infty}\frac{\Phi_{q}(K_{i})}{g({K_i})}. \]
Une famille $\Kcal=(K_{i})_{i\in\mathbb{N}}$ est dite \emph{asymptotiquement exacte} si  toutes les limites $\phi_{q}(\Kcal)$, pour $q\in\mathcal{Q}$, existent. En pratique, ce n'est pas une condition tr\`es restrictive. En effet, si $\Kcal=(K_{i})_{i\in\mathbb{N}}$  est une \emph{tour}, c'est-\`a-dire si $K_{i}\subset K_{i+1}$ pour tout $i$, alors $\Kcal$ est asymptotiquement exacte (voir \cite[\S2]{TV} o\`u il est de plus d\'emontr\'e que les invariants $\phi_q(\Kcal)$ ne d\'ependent que de l'union $\bigcup_i K_{i}$).  On montre plus g\'en\'eralement que, de toute famille $\Kcal$, on peut extraire une sous-famille asymptotiquement exacte.  Il est clair qu'une sous-famille  $\Kcal'$ d'une famille asymptotiquement exacte $\Kcal$ l'est \'egalement et que ses invariants $\phi_q(\Kcal')$ sont alors \'egaux \`a ceux de $\Kcal$. Nous renvoyons \`a \cite{LebTV} et \cite{LebTVE} pour une \'etude plus pouss\'ee des invariants $\phi_q(\Kcal)$.

Une famille asymptotiquement exacte $\Kcal$ est dite \emph{asymptotiquement mauvaise} si $\phi_q(\Kcal)=0$ pour tout $q\in\mathcal{Q}$, \emph{asymptotiquement bonne} sinon.  En pratique, les familles asymptotiquement mauvaises sont les plus communes. Pour tout corps de nombres $K$ et tout $q\in\mathcal{Q}_0$ il est clair que $\Phi_q(K)\leq [K:\Q] = \Phi_\R(K) + 2\Phi_\C(K)$. Si $(K_i)_{i\in\N}$ est une famille asymptotiquement exacte, il suit que le quotient $[K_i:\Q]/g(K_i)$ admet une limite (finie) et que cette limite est nulle si et seulement si la famille est asymptotiquement mauvaise. Une famille est donc asymptotiquement mauvaise si et seulement si elle satisfait la condition \eqref{ieq:asymp.bad} de l'introduction. \\

Si  $\Kcal$ est une famille asymptotiquement exacte, nous lui associons les deux quantit\'es 
\[ \beta_0(\Kcal) := \sum_{q\in\mathcal{Q}_0}\phi_{q}(\Kcal)\, \log\frac{q}{q-1} 
\quad \text{ et } \quad 
\beta(\Kcal) := 1+\beta_0(\Kcal)-\phi_{\R}(\Kcal)\,\log 2-\phi_{\C}(\Kcal)\,\log 2\pi.\]
La preuve du th\'eor\`eme 7.1 de \cite{TV} assure que la somme d\'efinissant $\beta_0(\Kcal)$ converge. Les th\'eor\`emes~J et L du m\^eme article fournissent des exemples de tours $\Kcal$ telles que $\beta(\Kcal)\neq 1$ tandis que la partie~8 montre que l'on a toujours 
\[2/5 \leq \beta(\Kcal)\leq 6/5\]
(les \'enonc\'es des th\'eor\`emes 8.2 et 8.3 ne pr\'esentent que les cons\'equences pour le quotient de Brauer--Siegel mais la m\'ethode de programmation lin\'eaire employ\'ee concerne bien $\beta(\Kcal)$).

\subsection{Th\'eor\`eme de Brauer--Siegel g\'en\'eralis\'e : architecture des preuves}
Supposons donn\'ee une famille asymptotiquement exacte $\Kcal = (K_i)_i$ de corps de nombres. Rappelons rapidement comment s'articule classiquement la preuve de la conjecture de Tsfasman--Vl\u{a}du\c{t} pour la famille $\Kcal$ dans les cas o\`u celle-ci est connue, tout en donnant une br\`eve esquisse de la preuve du th\'eor\`eme \ref{itheo:principal}. Il s'agit de montrer que la limite
\[\lim_{i\to\infty}\frac{\log\big(h(K_i) \, R(K_i)\big)}{\log\sqrt{|\Delta(K_{i})|}} \]
existe et vaut $\beta(\Kcal)$. Dans les notations introduites plus haut, la formule des classes de Dirichlet (voir \cite[VIII.\,\S2, th\'eor\`eme 5]{LangANT}) entra\^ine 
\begin{equation*}
\frac{ \log \zeta_{K_i}^\ast(1)}{g(K_i)} = \frac{\log\big(h(K_i)\, R(K_i)\big)}{g(K_i)} - 1 + \frac{\Phi_\R(K_i)}{g(K_i)}\,\log 2 
+ \frac{\Phi_\C(K_i)}{g(K_i)}\,\log 2\pi - \frac{\log \mu_{K_i}}{g(K_i)},
\end{equation*}
o\`u $\mu_{K_i}$ d\'esigne le nombre de racines de l'unit\'e contenues dans $K_i$. \`A l'aide de cette relation,  nous nous ramenons ainsi \`a prouver  
\begin{equation*}
\lim_{i\to\infty} \frac{\log\zeta^\ast_{K_i}(1)}{g(K_i)} = \beta_0(\Kcal). 
\end{equation*}
L'in\'egalit\'e de Tsfasman--Vl\u{a}du\c{t} (voir th\'eor\`eme 7.1 de \cite{TV}) donne 
\begin{equation}\label{eq:TV.ineq}
\limsup_{i} \frac{\log\zeta^\ast_{K_i}(1)}{g(K_i)} 
\leq \beta_0(\Kcal) <\infty. 
\end{equation}
Avec cette in\'egalit\'e,  d\'emontrer la conjecture de Tsfasman--Vl\u{a}du\c{t} pour la famille $\Kcal$ revient finalement \`a s'assurer que
\[ \liminf_{i} \frac{\log\zeta^\ast_{K_i}(1)}{g(K_i)}
\geq \beta_0(\Kcal). \] 
En des termes plus vagues, il s'agit d'obtenir une minoration asymptotique  suffisamment pr\'ecise du r\'esidu $\zeta^\ast_{K_i}(1)$ en termes de $g(K_i)$. Une telle minoration s'obtient assez rapidement lorsque la fonction $\zeta_{K_i} $ n'a pas de z\'ero r\'eel \og{}trop proche\fg{} de $1$ (voir lemme \ref{lemm:lemme6}). Dans le cas o\`u un tel z\'ero est pr\'esent, un travail suppl\'ementaire est n\'ecessaire. Notre approche consiste alors \`a utiliser un r\'esultat de \og{}descente de z\'eros\fg{} (th\'eor\`eme \ref{itheo:desc.gael}) qui, sous hypoth\`eses suppl\'ementaires sur $K_i$,   \og{}explique\fg{} la pr\'esence d'un tel z\'ero par l'existence d'un sous-corps quadratique de $K_i$   dont la fonction z\^eta s'annule en ce m\^eme point. Il restera alors \`a appliquer le th\'eor\`eme de Siegel \`a la suite de corps quadratiques ainsi construite.

\section{Th\'eor\`eme de Mertens et th\'eor\`eme de Brauer--Siegel} 

Soit $K$ un corps de nombres diff\'erent de $\Q$. Nous dirons que $K$ admet un \emph{z\'ero exceptionnel} s'il existe un zero r\'eel $\rho$ de sa fonction z\^eta $\zeta_K$ v\'erifiant 
\[ 1 - ({8} g(K))^{-1} \leq \rho < 1. \] 
Si un tel z\'ero existe il est unique et simple d'apr\`es le lemme 3 de \cite{Stark}. Dans la suite de cet article, nous noterons $\rho({K})$ le z\'ero exceptionnel de $\zeta_{K}$ s'il existe  et poserons $\rho(K)=0$ sinon. \\

L'article \cite{Lebacque} d\'emontre la forme effective (et inconditionnelle) suivante du th\'eor\`eme de Mertens.
\begin{theo}\label{theo:mertens}  
Il existe une constante effective $C>0$ telle que, pour tout corps de nombres $K$ et tout $X>1$ v\'erifiant $\log X>C [K:\Q] g(K)^2$,  
\begin{equation}\label{eq:mertens}
\sum_{\substack{q\in\mathcal{Q}_0 \\ q\leq X}}	\Phi_q(K) \, \log\frac{q}{q-1} = \log\log X + \gamma  + \log\zeta_K^\ast(1) + O\left( \frac{1}{(1-\rho(K))\log X} \right),
\end{equation}
o\`u $\gamma$ est la constante d'Euler--Mascheroni. La constante implicite dans le terme d'erreur est absolue et effective.
\end{theo}

Pour le confort de lecture, rappelons comment d\'eduire de ce th\'eor\`eme le lemme ci-dessous (qui est essentiellement le lemme 6 de \cite{Lebacque}). 
\begin{lemm}
\label{lemm:lemme6}
 Soit $\Lcal = (L_i)_{i\in\N}$ une famille asymptotiquement exacte de corps de nombres. Pourvu~que
\[\liminf_{i} \frac{\log(1-\rho(L_{i}))}{g({L_i})} =0, \qquad \text{{\it i.e.} que }\ \limsup_{i} \frac{|\log(1-\rho(L_{i}))|}{g({L_i})} =0, \]
la famille $\Lcal$ v\'erifie le th\'eor\`eme de Brauer--Siegel g\'en\'eralis\'e. 
\end{lemm}

\begin{proof} 
Avec  l'in\'egalit\'e de Tsfasman--Vl\u{a}du\c{t} \eqref{eq:TV.ineq},  il suffit de montrer que $\beta_0(\Lcal)$ v\'erifie
\[\beta_0(\Lcal) \leq \liminf_i \frac{\log\zeta^\ast_{L_i}(1)}{g(L_i)}.\]
Pour all\'eger les notations dans cette preuve, nous noterons $n_i=[L_i:\Q]$, $g_i=g(L_i)$ et $\rho_i=\rho(L_i)$ pour tout $i$. Pour tout $i$, posons $X_i :=  \exp\left(2C\, {n_i g_i^2}/{(1 - \rho_i)}\right)$ et, pour tout $q\in\mathcal{Q}_0$, 
\[f_i(q) :=  \frac{\Phi_q(L_i)}{g_i}\, \log\frac{q}{q-1} \ \text{ si } q\leq X_i \quad \text{ et } f_i(q) := 0 \ \text{ sinon}.\]
Comme $\rho_i\in[0, 1[$, $X_i$ v\'erifie  $\log X_i>C n_ig_i^2$. L'estimation \eqref{eq:mertens} entra\^ine que, pour tout $i\in\N$,
\[ \frac{\log\zeta_{L_i}^\ast(1) }{g_i} \geq  \sum_{q\in\mathcal{Q}_0} f_i(q) -\frac{\log\log X_i}{g_i} - \frac{\gamma}{g_i} - \frac{c_0}{ g_i (1-\rho_i)\log X_i}\]
pour une certaine constante absolue $c_0>0$. Par choix de $X_i$, nous tirons 
\begin{align}\label{eq:demo.lemm6.2} 
\frac{\log\zeta_{L_i}^\ast(1) }{g_i} &\geq \sum_{q\in\mathcal{Q}_0} f_i(q) + \frac{\log(1-\rho_i)}{g_i} - \frac{\gamma+\log(2C)}{g_i} - \frac{c_0 (2C)^{-1}}{n_ig_i^3}  - \frac{\log n_i}{g_i} - \frac{2\log g_i}{g_i}  \notag \\
&\geq \sum_{q\in\mathcal{Q}_0} f_i(q) + \frac{\log(1-\rho_i)}{g_i} - \frac{\gamma+\log(2C) +c_0 (2C)^{-1}+\log (n_i/g_i)}{g_i} - \frac{3\log g_i}{g_i}  \notag \\
&\geq \sum_{q\in\mathcal{Q}_0} f_i(q) -\frac{|\log(1-\rho_i)|}{g_i} - o(1) \qquad \text{lorsque } i\to\infty.
\end{align}
Rappelons en effet que $n_i/g_i$ est born\'e (par in\'egalit\'e de Minkowski \eqref{ineq:minkowski}) et que $(\log g_i)/g_i =o(1)$.

La famille $\Lcal$ \'etant asymptotiquement exacte, il est clair que 
\[ \sum_{q\in\mathcal{Q}_0} \liminf_i f_i(q)  =\sum_{q\in\mathcal{Q}_0} \lim_{i\to\infty} f_i(q)  = \beta_0(\Lcal)<\infty.\] 
Par ailleurs,  nous d\'eduisons du lemme de Fatou que 
\[ \beta_0(\Lcal)  = \sum_{q\in\mathcal{Q}_0} \liminf_i f_i(q) \leq \liminf_i \sum_{q\in\mathcal{Q}_0}  f_i(q).\]
En prenant les limites inf\'erieures de part et d'autre de \eqref{eq:demo.lemm6.2} et en reportant l'in\'egalit\'e ci-dessus dans la relation obtenue, nous obtenons 
\begin{equation*}
\liminf_i \frac{\log\zeta_{L_i}^\ast(1) }{g(L_i)} \geq  \beta_0(\Lcal) - \limsup_i\frac{|\log(1-\rho_i)|}{g_{i}} + 0.
\end{equation*}
L'hypoth\`ese assure alors que  $\ds \liminf_i \frac{\log\zeta_{L_i}^\ast(1) }{g(L_i)}  \geq \beta_0(\Lcal)$, ce qu'il fallait d\'emontrer. 
\end{proof}

On d\'eduit imm\'ediatement du lemme \ref{lemm:lemme6} ci-dessus le cas particulier important (et bien connu, voir th\'eor\`eme 7.2 et remarque 7.2 de \cite{TV}) suivant: {\it une famille asymptotiquement exacte de corps de nombres qui n'ont pas de z\'ero exceptionnel 
 v\'erifie le th\'eor\`eme de Brauer--Siegel g\'en\'eralis\'e.} Ainsi la conjecture de Tsfasman--Vl\u{a}du\c{t} suit de l'hypoth\`ese de Riemann g\'en\'eralis\'ee (condition \refcond{c:GRH} dans l'introduction).

\section{Th\'eor\`eme de descente de z\'eros}

\subsection{Complexit\'es galoisiennes} 
Si $L/K$ est une extension de corps de nombres et $M/K$ sa cl\^oture galoisienne, nous notons $[L:K]_{\Gal} := [M:K]$ et $\clgal_{L/K}$ le plus petit entier $n\geq 1$ tel que $M$ est le compositum de $n$ conjugu\'es de $L$ sur $K$. Nous posons aussi $[L:K]'_{\Gal} := \clgal_{L/K}\, [M:L]$. 	

Reprenons \`a pr\'esent la d\'efinition de la complexit\'e galoisienne donn\'ee dans l'introduction en lui adjoignant deux variantes.
\begin{defi}
Pour une extension $L/K$ de corps de nombres, nous notons 
\begin{align*}
\comp_{L/K} & := \min\left\{ \max_{1\leq i \leq s} [L_i:L_{i-1}]_{\Gal}, \ K=L_0 \subset L_1 \subset \dots \subset L_s=L\right\}, \\
	\comp'_{L/K} & 	:= \min\left\{ \max_{1\leq i \leq s} [L_i:L_{i-1}]'_{\Gal}, \ K=L_0 \subset L_1 \subset \dots \subset L_s=L\right\} \\
	\text{et }\quad 	
	\comp''_{L/K} & 	:= \min\left\{ \max_{1\leq i \leq s} [L_i:L_{i-1}]'_{\Gal}\, \min\big\{1, 2\, [L:L_i]^{-1}\big\}, \ K=L_0 \subset L_1 \subset \dots \subset L_s=L\right\}. 
\end{align*}
\end{defi}

Voici quelques propri\'et\'es imm\'ediates de ces invariants.
\begin{lemm}\label{lemm:comp.gal.relations}
Soit $L/K$ une extension  de corps de nombres.
\begin{enumerate}[(i)]
	\item 	$\displaystyle 1\leq \comp''_{L/K}\leq \comp'_{L/K}\leq \comp_{L/K} \leq [L:K]_{\Gal} \leq \frac{[L:K]!}{([L:K]-\clgal_{L/K})!} \leq [L:K]!$.
	\item 	$\displaystyle \comp_{L/K} \leq [L:K]\,\comp''_{L/K}$.
	\item 	$L/K$ est galoisienne si et seulement si $\clgal_{L/K}=1$ si et seulement si $[L:K]'_{\Gal} = 1$.  
	\item 	$L/K$ est galoisienne par pas si et seulement si $\comp'_{L/K}=1$, auquel cas nous avons $\comp_{L/K}\leq [L:K]$.
	\item\label{ineq:comp.sousext} 	Si $N$ est un corps de nombres contenant $L$ alors $\comp_{N/K}\leq \max\{\comp_{N/L},\, \comp_{L/K}\}$ et de m\^eme pour $\comp'$ et~$\comp''$.
\end{enumerate}
\end{lemm}

Voyons ensuite que la complexit\'e galoisienne n'est pas trop grande en pr\'esence d'une extension  r\'esoluble.
\begin{lemm}\label{lemm:compl.gal.clotgalres} 
Soient $K\subset L \subset N$ et $K\subset M \subset N$ des extensions de corps de nombres telles que $M/K$ et $N/K$ sont galoisiennes et $N/M$ est r\'esoluble.  Si $\comp_{L/K} > [M:K]$,  nous avons  
\begin{enumerate}[(i)]
	\item\label{lemm:cg.i}		$\displaystyle \comp'_{L/K} \leq \max_{p \text{ premier}}\left\{ (v_p([L:K])+1)\, p^{v_p([L:K])^2} \right\}$,		 
	\item\label{lemm:cg.ii}		$\displaystyle \comp_{L/K} \leq \max_{p \text{ premier}}\left\{ p^{v_p([L:K])^2+v_p([L:K])}\right\}$ 				 
		et
	\item\label{lemm:cg.iii} 	$\log\comp_{L/K} \leq 3\, (\log [L:K])^2$.
\end{enumerate}	
\end{lemm}

\begin{proof}
Commen\c{c}ons par \'etablir l'\'enonc\'e dans le cas particulier o\`u l'extension $L/K$ n'admet aucun sous-corps interm\'ediaire. Il n'y a pas de restriction \`a supposer en outre que $N/K$ est la cl\^oture galoisienne de $L\cdot M /K$. Puisque 
\[[M:K]<  \comp_{L/K} \leq [L:K]_{\Gal} \leq [N:K],\]
nous avons $M\neq N$. Notre argument s'inspire alors de celui de \cite[p. 517]{MurtyMRL},  lui-m\^eme adapt\'e des preuves des th\'eor\`eme 3 et lemme 2 de \cite{OS93}. Notons $G = \Gal(N/K)$, $H=\Gal(N/L)$ et \linebreak $\Gamma = \Gal(N/M)$.  Nos hypoth\`eses se traduisent par les faits suivants : $\Gamma$ est un sous-groupe distingu\'e de $G$,  $\Gamma$ est un groupe r\'esoluble non trivial,  il n'y a pas de sous-groupe interm\'ediaire entre $H$ et $G$ et le seul sous-groupe distingu\'e de $G$ contenu dans $\Gamma\cap H$ est  trivial. Consid\'erons l'ensemble $\mathscr E$ des sous-groupes non triviaux de $\Gamma$ qui sont distingu\'es dans $G$. Cet ensemble est non vide car $\Gamma\in\mathscr E$.  Fixons un \'el\'ement $A\in\mathscr E$ minimal pour l'inclusion. Tout sous-groupe caract\'eristique de $A$ est distingu\'e dans~$G$. Par minimalit\'e, le seul sous-groupe caract\'eristique non trivial de $A$ est $A$ lui-m\^eme. Puisque $A$ est r\'esoluble (comme sous-groupe de $\Gamma$), ceci entra\^ine que $A$ est ab\'elien \'el\'ementaire  c'est-\`a-dire isomorphe \`a $(\Z/p\Z)^r$ pour un nombre premier $p$ et $r\geq 1$.

Montrons que $G$ s'\'ecrit comme le produit semi-direct $G=A\rtimes H$. Tout d'abord notons que $A\not\subset H$ car sinon $A$ serait un sous-groupe non trivial de $\Gamma\cap H$ distingu\'e dans $G$. Le sous-ensemble $A\cdot H$ est un sous-groupe de $G$ (car $A$ est distingu\'e) qui contient donc strictement $H$. Nous avons ainsi $G=A\cdot H$.  D'autre part, l'intersection $A\cap H$ est un sous-groupe distingu\'e \`a la fois de $H$ (puisque $A$ est distingu\'e dans $G$) et de $A$ (puisque $A$ est ab\'elien)  donc de $A\cdot H = G$. L'inclusion $A\cap H \subset \Gamma\cap H$ force ainsi $A\cap H$ \`a \^etre trivial  et nous avons bien $G = A \rtimes H$. En particulier  l'ordre $p^r$ de $A$ s'interpr\`ete  comme l'indice de $H$ dans $G$, c'est-\`a-dire comme le degr\'e de $L/K$.

Notons $\varphi : H\to \mathrm{Aut}(A)$ l'action de $H$ sur $A$. Pour tout $a\in A$, l'\'ecriture de $G$ comme produit semi-direct montre que $H \cap a^{-1}Ha = \{h\in H \mid \varphi(h)(a)=a\}$. Par suite, si $(a_1, \dots, a_r)$ d\'esigne une base de $A$, 
\[ H\cap \bigcap_{i=1}^r a_i^{-1}Ha_i = \ker\varphi.\] 
Ce noyau est un sous-groupe distingu\'e de $H$ contenu dans le centralisateur de $A$ donc il est  distingu\'e dans $A\cdot H = G$. Par correspondance de Galois, le corps $N^{\ker\varphi}$ est une extension galoisienne de $K$ qui s'\'ecrit comme compositum de $r+1$ conjugu\'es de $L$ au-dessus de $K$. En particulier $N^{\ker\varphi}$ n'est autre que la cl\^oture galoisienne de $L/K$ et nous avons $\clgal_{L/K}\leq r+1$. Nous en d\'eduisons 
\[ \comp'_{L/K} = \clgal_{L/K}\, [N^{\ker\varphi}:L ]  = \clgal_{L/K}\, [H:\ker\varphi] = \clgal_{L/K}\, |\mathrm{Im}\, \varphi|\]
et, de m\^eme, 
$\comp_{L/K} = [L:K]\, |\im\varphi|$. Puisque $\im\varphi$ est un sous-groupe de $\mathrm{Aut}(A)\simeq\mathrm{GL}_r(\Z/p\Z)$, son cardinal est au plus $p^{r^2}$. Il vient $\comp'_{L/K}\leq (r+1)\, p^{r^2}$, $\comp_{L/K}\leq p^{r^2+r}$ et 
\[ \log\comp_{L/K}\leq (r^2+r)\,\log p \leq 2 r^2\,\log p \leq 3r^2\, (\log p)^2 = 3\,(\log [L:K])^2.\]
Le lemme est donc d\'emontr\'e lorsque $L/K$ n'a pas d'extension interm\'ediaire.

Pour passer au cas g\'en\'eral il suffit de choisir une suite $K=L_0 \subset L_1 \subset \dots \subset L_s= L$ o\`u aucune des extensions $L_i/L_{i-1}$ n'admet d'extension interm\'ediaire, en remarquant que $L_{i-1}\subset L_i \subset N$ et $L_{i-1}\subset L_i \cdot M \subset N$ satisfont les hypoth\`eses galoisiennes de l'\'enonc\'e avec $[L_i\cdot M : L_{i-1}] \leq [M:K]$. 
\end{proof}

\subsection{Genre d'un compositum}
Rappelons d'abord un \'enonc\'e classique.
\begin{lemm}\label{lemm:genre.compositum}
Soient $L_1, \dots, L_n$ des corps de nombres et $M$ leur compositum. Alors
\begin{equation*}
	\frac{g(M)}{[M:\Q]} \leq \sum_{i=1}^n \frac{g(L_i)}{[L_i:\Q]}.
\end{equation*}
\end{lemm}

\begin{proof}
	Ceci d\'ecoule directement de la relation de divisibilit\'e entre discriminants donn\'ee par le lemme 7 de \cite{Stark}.
\end{proof}

Nous utiliserons la cons\'equence suivante.
\begin{lemm}\label{lemm:inegalite.genre}
Soient $K \subset L\subset M$ des corps de nombres imbriqu\'es et $N/K$ la cl\^oture galoisienne de $L/K$. Alors 
\[g(M\cdot N) \leq [L:K]'_{\Gal}\, g(M).\]
\end{lemm}

\begin{proof}
Par d\'efinition, $N$ est le compositum de $\clgal_{L/K}$ conjugu\'es de $L$ au-dessus de $K$. Par suite, $M\cdot N$ est le compositum de $M$ et de $\clgal_{L/K}-1$ de ces conjugu\'es (qui ont tous le m\^eme genre). Le lemme pr\'ec\'edent donne donc
\[\frac{g(M\cdot N)}{[M\cdot N : \Q]} \leq \big(\clgal_{L/K} - 1\big)\, \frac{g(L)}{[L:\Q]} + \frac{g(M)}{[M:\Q]}.\]
Puisque $[M:L]\, g(L) \leq g(M)$, il vient  
\[g(M\cdot N)\leq \clgal_{L/K}\, [M\cdot N : M]\, g(M) \leq \clgal_{L/K}\, [N:L] \, g(M).\]
C'est le r\'esultat souhait\'e, par d\'efinition de $[L:K]'_{\Gal}$.
\end{proof}

\subsection{Extensions quadratiques}
Avant de passer \`a la descente, \'etablissons un dernier lemme pr\'eliminaire (utilisant d\'ej\`a les r\'esultats de Brauer et Stark qui seront cruciaux pour la d\'emonstration du th\'eor\`eme \ref{theo:desc.gael} ci-dessous).
\begin{lemm}\label{lemm:desc.multiquad}
Soient $L_1, \dots, L_n$ des extensions au plus quadratiques d'un corps de nombres $K$ et $\rho$ un r\'eel  tel que 
\[1 - \Big(32\, \max_{1\leq i\leq n} g(L_i)\Big)^{-1} \leq \rho < 1.\]
Si $\zeta_{L_i}(\rho)=0$ pour tout $1\leq i \leq n$	alors $\zeta_{L_1\cap \dots \cap L_n}(\rho) = 0$.
\end{lemm}

\begin{proof}
Si $L_1\cap \dots\cap L_n$ est \'egal \`a l'un des $L_i$ alors la conclusion est claire. Sinon il existe deux indices $1\leq i <j\leq n$ avec $L_i\neq L_j$ et $[L_i:K]=[L_j:K]=2$. Le compositum $M = L_i\cdot L_j$ est alors une extension biquadratique (donc galoisienne) de $K$  dont le genre v\'erifie  $g(M)\leq 2(g(L_i) + g(L_j))$ d'apr\`es le lemme \ref{lemm:genre.compositum}.
	
Le th\'eor\`eme 1 de \cite{Brauer47} dans l'extension $M/L_i$ assure que $\rho$ est un z\'ero de $\zeta_M$.  Comme $\rho$ appartient \`a l'intervalle $\big[1-(8\, g(M))^{-1} , 1\big[$,  c'est un z\'ero simple de $\zeta_M$. Nous appliquons alors le th\'eor\`eme 3 de \cite{Stark} \`a l'extension $M/K$ et constatons que son sous-corps minimal dont la fonction z\^eta s'annule en $\rho$, puisqu'il doit \^etre contenu dans $L_i$ et $L_j$, ne peut \^etre que $K = L_1\cap \dots \cap L_n$.
\end{proof}

\subsection{Descente de zéros}
Nous d\'emontrons ici une version renforc\'ee (puisque $\comp''_{L/K}\leq \comp_{L/K}$) du th\'eor\`eme \ref{itheo:desc.gael} de l'introduction.
\begin{theo}\label{theo:desc.gael}
Soit $L/K$ une extension de corps de nombres. Si la fonction z\^eta de $L$ admet un z\'ero r\'eel $\rho$ v\'erifiant
\[ 1 - \left(32\, g(L)\, \comp''_{L/K}\right)^{-1} \leq \rho <1,\]
alors il existe une sous-extension $F$ de $L/K$ avec $[F:K]\leq 2$ et $\zeta_F(\rho)=0$.
\end{theo}

\begin{proof}
Nous raisonnons par r\'ecurrence sur le degr\'e de l'extension $L/K$. Si ce dernier est $1$ ou $2$, il n'y a rien \`a faire : le choix $F=L$ convient. Supposons donc $[L:K]\geq 3$ et le r\'esultat vrai pour toute extension  de degr\'e au plus $[L:K]-1$. Par d\'efinition de la complexit\'e galoisienne $\comp''_{L/K}$, nous pouvons choisir une suite de sous-extensions 
\[K = L_0 \subset L_1\subset L_2\subset\dots\subset L_s = L\]
de sorte que $\comp''_{L/K} = \displaystyle\max_{1\leq i \leq s}\left\{ [L_i:L_{i-1}]'_{\Gal} \, \min\big(1, 2\, [L:L_i]^{-1}\big)\right\}$.  Il n'y a pas de restriction \`a supposer  $[L_i:L_{i-1}]\geq 2$ pour $1\leq i\leq s$. Posons $s'=s$ si $[L:L_{s-1}]\geq 3$ et $s'=s-1$ si $[L:L_{s-1}]=2$.   Notons alors $K'=L_{s'-1}$ et $L'=L_{s'}$. Le choix de $s'$ montre que nous avons $K \subset K'\subset L' \subset L$, 
\[[L:L'] \leq 2, \quad [L:K']> 2 \quad \text{ et } \quad [L':K']'_{\Gal} \leq \comp''_{L/K}.\]
D\'esignons \`a pr\'esent par $M'/K'$ la cl\^oture galoisienne de $L'/K'$.  D'apr\`es le lemme \ref{lemm:inegalite.genre}, nous avons
\begin{equation*}
	g(L\cdot M')\leq [L':K']'_{\Gal} \, g(L) \leq \comp''_{L/K}\, g(L).
\end{equation*}
Ce compositum $L\cdot M'$ est une extension au plus quadratique de $M'$ tandis que c'est une extension galoisienne de $L$. Le th\'eor\`eme 1 de \cite{Brauer47} assure   $\zeta_{L\cdot M'}(\rho) =0$. Tout comme la borne  sur le genre ci-dessus, cette annulation vaut \'egalement pour chacun des conjugu\'es de $L\cdot M'$ sur $K'$. Par cons\'equent, si nous désignons par $N$ l'intersection de tous ces conjugu\'es (qui est une extension galoisienne de $K'$), le lemme \ref{lemm:desc.multiquad} assure   $\zeta_N(\rho) =0$.

Puisque $N$ est un sous-corps de $L\cdot M'$, nous avons $g(N)\leq \comp''_{L/K}\, g(L)$ donc $\rho > 1 - (32\, g(N))^{-1}$ et, en particulier, $\rho$ est un z\'ero simple de $\zeta_N$. Le th\'eor\`eme 3 de \cite{Stark} montre alors qu'il existe une sous-extension $F'/K'$ de $N/K'$ telle que $[F':K']\leq 2$, $\zeta_{F'}(\rho)=0$ et qui est minimale pour cette annulation au sens o\`u, pour toute sous-extension $N'/K'$ de $N/K'$, on a 
\[F'\subset N' \ \iff \ \zeta_{N'}(\rho)=0.\] 
En particulier, la fonction z\^eta du compositum $F'\cdot L'$ s'annule en $\rho$. Ce compositum $F'\cdot L'$ est une extension au plus quadratique de $L'$ et (comme sous-corps de $L\cdot M'$) son genre est au plus $\comp''_{L/K}\, g(L)$. Ces trois propri\'et\'es sont partag\'ees par $L$ donc le lemme \ref{lemm:desc.multiquad} entra\^ine $\zeta_{L\cap F'\cdot L'}(\rho)=0$. Nous en d\'eduisons   $F'\subset L\cap  F'\cdot L'  \subset L$.

\begin{center}
\begin{tikzpicture}
    \node (Q1) at (0,0) {$K'$};
    \node (Q2) at (-2,1) {$L'$};
    \node (Q3) at (2,1) {$F'$};
    \node (Q4) at (0,2) {$F'\cdot L'$};        
    \node (Q5) at (-4,2) {$L$};
    \node (Q6) at (-2,3) {$N$};
	\node (Q7) at (-4,4) {$L\cdot M'$};

    \draw (Q1)--(Q2) ;
    \draw (Q1)--(Q3) [style= very thick] node [pos=0.5, below, inner sep=1mm, rotate=25]  {\footnotesize $\leq 2$};
    \draw (Q2)--(Q4) [style= very thick] node [pos=0.6, below, inner sep=1mm, rotate=25]  {\footnotesize $\leq 2$};
    \draw (Q2)--(Q5) [style= very thick] node [pos=0.5, below, inner sep=1mm, rotate=-30]  {\footnotesize $\leq 2$};
    \draw (Q3)--(Q4) ; 
    \draw (Q4)--(Q6); 
    \draw (Q5)--(Q7); 
    \draw (Q6)--(Q7); 
    \end{tikzpicture}
\end{center}
Cette inclusion fournit $g(F')\leq g(L) \, [L:F']^{-1}$. D'un autre c\^ot\'e, la suite 
\[K=L_0\subset L_1\subset \dots \subset L_{s'-1}=K'\subset F'\]
montre   $\comp''_{F'/K}\leq [L:F']\, \comp''_{L/K}$ car $[F':K']'_{\Gal} = 1$ et, si $1\leq i\leq s'-1$, 
\[[L_i:L_{i-1}]'_{\Gal}\, \min\left\{1, \frac{2}{[F':L_i]}\right\}
\leq [L_i:L_{i-1}]'_{\Gal}\, \min\left\{1, \frac{2}{[L:L_i]}\right\}\, [L:F']
\leq [L:F']\, \comp''_{L/K}.\]
Nous avons ainsi $[F':K]\leq 2[K':K]<[L:K]$, $\zeta_{F'}(\rho)= 0$ et $\comp''_{F'/K}\, g(F')\leq \comp''_{L/K}\, g(L)$.

L'hypoth\`ese de r\'ecurrence montre donc bien l'existence d'un corps $F$ avec $K\subset F\subset F'\subset L$, $[F:K]\leq 2$ et $\zeta_F(\rho)=0$.
\end{proof}

\subsection{Commentaires}\label{ss:commentaires}
Dans la partie suivante, nous n'utiliserons le th\'eor\`eme \ref{theo:desc.gael} que lorsque ${K=\Q}$. Dans ce cas sa conclusion est renforc\'ee en \og{}$[F:\Q]=2$\fg{} puisque la fonction z\^eta de Riemann $\zeta_\Q$ ne s'annule  pas sur l'intervalle $[0,1[$.

Notre r\'esultat \og{}explique\fg{} par l'usage de la complexit\'e galoisienne les r\'esultats ant\'erieurs de \og{}descente de z\'eros\fg{} dus \`a Stark et Murty. En effet, si nous majorons $\comp''_{K/\Q}$ par $[K:\Q]!$, nous voyons appara\^itre le lemme 8 de \cite{Stark} (\`a un facteur $4$ pr\`es);  si nous remarquons que, si $K$ est un corps CM  de degr\'e $2n$, nous avons $\comp''_{K/\Q}\leq n!$, nous sommes dans le cadre du lemme 9 de ce m\^eme article  tandis que l'\'egalit\'e $\comp''_{K/\Q}=1$ pour une extension galoisienne par pas correspond \`a son lemme 10. Dans les deux cas, nous pourrions aussi retrouver la pr\'ecision sur le caract\`ere imaginaire de $F$ en reprenant notre d\'emonstration (qui se r\'eduit \`a celle de Stark dans ces situations particuli\`eres).

Dans le cas d'une extension $L/K$ dont la cl\^oture galoisienne est r\'esoluble, la majoration \itref{lemm:cg.i} du lemme \ref{lemm:compl.gal.clotgalres} montre que notre r\'esultat est plus pr\'ecis que le  th\'eor\`eme 2.1 de \cite{Murty01} puisque $\comp''_{L/K}$ y est remplac\'e par 
\[c\, (12[L:K])^{\max_p v_p([L:K])}\, (1+{\max_p v_p([L:K])})^2,\]
o\`u $c$ est une constante non pr\'ecis\'ee (et bien entendu nous aurions pu autoriser comme Murty une partie imaginaire \`a $\rho$ puisque le lemme 3 de \cite{Stark} entra\^ine imm\'ediatement que $\rho$ est r\'eel). Les m\^emes remarques s'appliquent au th\'eor\`eme 2.2 de \cite{Wong22}, combinaison des \'enonc\'es de Stark et Murty, dans lequel $\comp''_{L/\Q}\leq \comp''_{L/K}$ est \'egalement major\'e par notre lemme \ref{lemm:compl.gal.clotgalres} (et nous lisons plut\^ot $|\mathrm{Im}(s)|$ que $|\mathrm{Re}(s)|$ dans les domaines (2.1), (2.3) et (2.4) de \cite{Wong22} afin d'\'eviter qu'ils ne soient vides).

Enfin, en utilisant notre th\'eor\`eme \ref{theo:desc.gael} en lieu et place de ses lemmes 8 et 10 dans la d\'emonstration du th\'eor\`eme 1 de \cite{Stark}, nous en obtenons directement une version raffin\'ee o\`u $g(n)$ est remplac\'e par la quantit\'e plus petite $\comp''_{k/\Q}$.  La m\^eme chose vaut pour le th\'eor\`eme 2 car dans sa d\'emonstration la pr\'ecision que $F$ est imaginaire dans les lemmes 9 et 10 n'est en fait pas utilis\'ee.

\section{D\'emonstration du th\'eor\`eme principal}

Commen\c{c}ons par introduire la notion commode suivante :
\begin{defi}\label{def:CGM'}
Une famille $(K_i)_{i\in\N}$ de corps de nombres est dite \emph{de complexit\'e galoisienne mod\'er\'ee} si le quotient $(\log\comp_{K_i/\Q})/g(K_i)$ tend vers $0$ lorsque $i$ tend vers $+\infty$.
\end{defi}
Au vu des in\'egalit\'es entre ces quantit\'es, remplacer $\comp$ par $\comp'$ ou $\comp''$ ne changerait pas la d\'efinition.
\\

Nous d\'emontrons dans cette partie l'\'enonc\'e 
annonc\'e comme th\'eor\`eme \ref{itheo:principal} dans l'introduction : 
\begin{theo}\label{theo:principal}
Toute famille asymptotiquement exacte de corps de nombres de complexit\'e galoisienne mod\'er\'ee v\'erifie le th\'eor\`eme de Brauer--Siegel g\'en\'eralis\'e. 
\end{theo}

\begin{proof} 
Soit $(K_i)_{i\in\N}$ une telle famille. Pour chaque entier $i\geq 0$  posons $c_i := 32\, g(K_i)\, \comp_{K_i/\Q}$ et distinguons  deux cas : {\it ou bien la fonction z\^eta de $K_i$ ne s'annule pas sur l'intervalle $[1 - c_i^{-1}, 1[$}, auquel cas nous posons $F_i=\Q(\sqrt{-3})$ (ainsi $g(F_i)\leq g(K_i)$ par \eqref{ineq:minkowski});  {\it ou bien la fonction z\^eta de $K_i$ admet un z\'ero $\rho$ dans l'intervalle $[1 - c_i^{-1}, 1[$} et nous construisons un corps quadratique $F_i$ comme suit. Le th\'eor\`eme \ref{theo:desc.gael} de descente de z\'eros implique l'existence d'un corps quadratique $F_i/\Q$ contenu dans~$K_i$ tel  que $\zeta_{F_i}(\rho)=0$.  Comme $F_i$ est contenu dans $K_i$, nous avons $g(F_i) \leq {g(K_i)}/{[K_i:F_i]} \leq g(K_i)$. Dans ce cas $\rho$ est le z\'ero exceptionnel de $K_i$ et de $F_i$. Dans les deux cas, nous avons
\begin{equation}\label{eq:principal.1}
	\frac{\left|\log(1-\rho(K_i))\right|}{g(K_i)} 
\leq  \max\left\{\frac{\log c_i}{g(K_i)}, \frac{|\log (1-\rho(F_i))|}{g(K_i)}\right\}.
\end{equation} 
L'hypoth\`ese de complexit\'e galoisienne mod\'er\'ee faite sur $\Kcal$ implique que $(\log c_i)/g(K_i)$ tend vers $0$ lorsque $i$ tend vers $+\infty$. Montrons \`a pr\'esent que le terme ${|\log (1-\rho(F_i))|}/{g(K_i)}$ au second membre de \eqref{eq:principal.1} est aussi de limite nulle.  Comme l'a remarqu\'e Walfisz (voir \cite{Walfisz36}) le th\'eor\`eme  de Siegel sur les corps quadratiques (voir \cite{Siegel35}) assure que, pour tout $\varepsilon>0$, il existe une constante $c(\varepsilon)>0$ telle que la fonction $\zeta_{F_i}$ ne s'annule pas sur l'intervalle
$\big[1-c(\varepsilon)\,|\Delta(F_i)|^{-\varepsilon}, 1\big[$. Par d\'efinition de $\rho(\cdot)$, on en d\'eduit l'existence d'une constante $A_{\varepsilon}\in\R$ telle que, pour tout $i\in\N$, $\left|\log(1-\rho(F_i))\right|\leq \varepsilon\,  g(F_i) + A_\varepsilon$. Pour tout $i\in\N$, nous avons ainsi
\[\frac{|\log(1-\rho(F_i))|}{g(K_i)} \leq \varepsilon \, \frac{g(F_i)}{g(K_i)} + \frac{A_\varepsilon}{g(K_i)} \leq \varepsilon  + \frac{A_\varepsilon}{g(K_i)}.\]
Puisque $g(K_i)$ tend vers $+\infty$ avec $i$, nous concluons que la quantit\'e $|\log(1-\rho(F_i))|/g(K_i)$ tend  vers~$0$. En reportant dans \eqref{eq:principal.1} les majorations obtenues,  il vient 
\[\lim_{i\to\infty} \frac{\left|\log(1-\rho(K_i))\right|}{g(K_i)} =0.\]
Le lemme \ref{lemm:lemme6} permet \`a pr\'esent de conclure.
\end{proof}

\begin{rema}
Voici un exemple de famille de corps de nombres qui n'est pas de complexit\'e galoisienne mod\'er\'ee. Notons $(p_i)_{i\geq 1}$ la suite croissante des nombres premiers et, pour chaque entier $i\geq 1$, $K_i$ le corps de rupture du polyn\^ome $X^{p_i}-X-1 \in\Z[X]$. L'extension $K_i/\Q$ est de degr\'e $p_i$  et sa cl\^oture galoisienne a pour groupe de Galois $\mathfrak{S}_{p_i}$ (voir les th\'eor\`eme 1 de \cite{Selmer56} et corollaire 3 de \cite{Osada87}). Puisque $K_i/\Q$ n'a pas de sous-corps interm\'ediaire, nous avons \[ \comp_{K_i/\Q} = [K_i:\Q]_{\Gal} = p_i!\] pour tout $i\geq 1$. D'autre part, les calculs sous le lemme 2 dans \cite{Osada87} impliquent que $|\Delta(K_i)|\sim {p_i}^{p_i}$ si bien que $g(K_i) \sim \frac{p_i}{2} \log p_i$ lorsque $i$ tend vers $+\infty$. Par cons\'equent le quotient $(\log \comp_{K_i/\Q})/g(K_i)$ tend vers $1/2$ lorsque $i$ tend vers $+\infty$.  

Nous ne connaissons pas d'exemple de tour de corps de nombres qui ne soit pas de complexit\'e galoisienne mod\'er\'ee.
\end{rema}

\section{Applications et exemples}\label{sec:exemples}

Dans cette partie, nous montrons que certaines des familles de corps de nombres apparaissant naturellement dans la th\'eorie asymptotique des corps globaux v\'erifient inconditionnellement le th\'eor\`eme de Brauer--Siegel g\'en\'eralis\'e.

\subsection{Extensions pro-r\'esolubles} 
Pour un corps de nombres $K$, on note $K\solv$ le compositum de toutes les extensions galoisiennes finies de $K$ dont le groupe de Galois  est r\'esoluble. 
\begin{theo}\label{pro.res} 
Soit $K$ un corps de nombres.  Toute famille asymptotiquement exacte de corps de nombres contenus dans $K\solv$ v\'erifie le th\'eor\`eme de Brauer--Siegel g\'en\'eralis\'e.
\end{theo}

\`A notre connaissance, ce r\'esultat est nouveau hormis dans le cas o\`u $K=\Q$  et seulement pour des familles asymptotiquement mauvaises ou des tours (voir \cite{Dixit}).

\begin{proof} 
Soit $\Lcal=(L_{i})_{i\in\N}$ une famille asymptotiquement exacte de corps de nombres contenus dans $K\solv$. Quitte \`a remplacer $K$ par sa cl\^oture galoisienne sur $\Q$ (ce qui ne fait qu'affaiblir l'hypoth\`ese sur $\Lcal$), nous pouvons supposer   l'extension $K/\Q$   galoisienne. Pour tout $i\in\N$, notons $N_i$ la cl\^oture galoisienne de l'extension $K\cdot L_i/K$. L'extension $N_i/K$ est alors r\'esoluble. Nous pouvons appliquer le lemme \ref{lemm:compl.gal.clotgalres} aux extensions $\Q\subset L_i\subset N_i$ et $\Q\subset K \subset N_i$. Celui-ci fournit la majoration :
\[\log \comp_{L_i/\Q} \leq \max\big\{ [K:\Q], 3\,(\log [L_i:\Q])^2\big\}.\]
Il est alors clair que le quotient $(\log \comp_{L_i/\Q})/g(L_i)$ tend vers $0$ lorsque $i$ tend vers $+\infty$. La famille $\Lcal$ est donc  de complexit\'e galoisienne mod\'er\'ee ce qui, avec le th\'eor\`eme \ref{theo:principal}, ach\`eve la preuve.
\end{proof}

En guise d'illustration du r\'esultat ci-dessus, proposons l'exemple d'application suivant. \'Etant donn\'e un corps de nombres $K$, notons $(H_i)_{i\in\N}$ la {\it tour de corps de Hilbert construite sur $K$} d\'efinie par r\'ecurrence par $H_0=K$ et, pour tout entier $i\geq 0$, $H_{i+1}$ est le corps de Hilbert de $H_i$. Nous posons alors $\Hcacal_K := \bigcup_{i\geq 0} H_i$.
\begin{coro}\label{coro:corps.hilbert}
Soit $K$ un corps de nombres. 
Toute famille asymptotiquement exacte de corps de nombres contenus dans $\Hcacal_K$
v\'erifie le th\'eor\`eme de Brauer--Siegel g\'en\'eralis\'e.
\end{coro}

\begin{proof}
L'extension $\Hcacal_K/K$ est pro-r\'esoluble (et non ramifi\'ee). Il suffit donc d'appliquer le th\'eor\`eme ci-dessus. 
\end{proof}

\subsection{Extensions $K_S(p)$}\label{ssec:KSp}  
Soient $K$ un corps de nombres,  $p$ un nombre premier et $S$ un ensemble de places de $K$. On consid\`ere  la pro-$p$-extension maximale $K_S(p)$ de $K$ qui est non ramifi\'ee hors de~$S$. Lorsque $S$ est fini et que l'extension $K_S(p)/K$ est infinie, on dit que le corps $K$ \emph{admet une $S$-$p$-tour de corps de classes infinie} (la seule fa\c{c}on connue de v\'erifier que la pro-$p$-extension $K_S(p)/K$ est infinie est le crit\`ere de Golod--Shafarevich, voir th\'eor\`eme 3 de \cite{HM1}). Les exemples d'extensions infinies asymptotiquement bonnes d'un corps de nombres donn\'e sont rares. Ceux-ci sont essentiellement construits \`a partir de sous-extensions de $K_{S}(p)$ pour un corps de nombres $K$ admettant une $S$-$p$-tour de classes infinie. 

Nous obtenons une version plus g\'en\'erale de la proposition \ref{prop:Sptour}, comme suit:

\begin{coro}\label{coro:pro.plusieurs} 
Soient $K_1, \dots, K_r$ des corps de nombres. Pour tout $j\in\{1, \dots, r\}$, nous fixons un nombre premier $p_j$ et un ensemble $S_j$ de places de $K_j$ et nous notons $\mathcal{K}$ le compositum 
\[\mathcal{K} := K_{1, S_1}(p_1)\cdot K_{2, S_2}(p_2)\cdot\ldots\cdot K_{r, S_r}(p_r).\] 
Supposons que $\mathcal K$ est de degr\'e infini sur $\Q$. Toute famille asymptotiquement exacte de corps de nombres contenus dans~$\mathcal{K}$ v\'erifie le th\'eor\`eme de Brauer--Siegel g\'en\'eralis\'e.

Si de plus, pour tout $j$, l'ensemble $S_j$ est fini et les places de $K_j$ contenues dans $S_{j}$ ne divisent pas $p_{j}$, une telle famille est asymptotiquement bonne.  
\end{coro}

\begin{proof} 
Soit $\Lcal=(L_i)_{i\in\N}$ une famille asymptotiquement exacte de corps de nombres tous  contenus dans $\mathcal{K}$. Le th\'eor\`eme \ref{pro.res} montre que $\Lcal$ v\'erifie le th\'eor\`eme de Brauer--Siegel g\'en\'eralis\'e puisqu'un pro-$p_j$-groupe est  en particulier  pro-r\'esoluble.  Supposons \`a pr\'esent que, pour tout $j$, $S_j$ est fini et que les places de $S_{j}$ ne divisent pas $p_{j}$. Soit $K$ le compositum des $K_{j}$.  Sous nos hypoth\`eses, l'extension $\mathcal{K}/K$ s'\'ecrit comme compositum d'extensions mod\'er\'ement ramifi\'ees donc est elle-m\^eme mod\'er\'ement ramifi\'ee. On note $S$ l'ensemble des places de $K$ divisant l'une de celles des $S_{j}$. Pour tout~$i$, la formule de Riemann--Hurwitz (voir \cite[chapitre III, proposition 3.13]{Neukirch}) appliqu\'ee \`a l'extension $K\cdot L_i/K$ montre 
\[ g(K\cdot L_{i})
\leq [K\cdot L_{i}:K]\left(g(K)+\frac{1}{2}\sum_{\mathfrak{p}\in S}\log \mathrm{N}\mathfrak{p}\right).\]
Il vient alors 
\[ \frac{g(L_{i})}{[L_{i}:\Q]}
\leq \frac{g(K\cdot L_{i})}{[K\cdot L_{i}:\Q]}\leq \frac{g(K)}{[K:\Q]} +\frac{\sum_{\mathfrak{p}\in S}\log\mathrm{N}\mathfrak{p} }{2[K:\Q]}. \]
Par suite, le quotient $g(L_i)/[L_{i}:\Q]$ reste born\'e lorsque $i$ varie donc son inverse $[L_i:\Q]/g(L_i)$ ne saurait tendre vers $0$ lorsque $i$ tend vers $+\infty$. Comme la famille $\Lcal$ est asymptotiquement exacte,  nous concluons comme au paragraphe \ref{ss:invariants} qu'elle est asymptotiquement bonne.
\end{proof}

Lorsque $r=1$, le corollaire pr\'ec\'edent implique ainsi que  \emph{toutes les extensions infinies $\mathcal K$ contenues dans une  $S$-$p$-tour de corps de classes   d'un corps de nombres   v\'erifient le th\'eor\`eme de Brauer--Siegel g\'en\'eralis\'e}.

 \subsection{Tours de corps de nombres}
Rappelons qu'une tour de corps de nombres est toujours asymptotiquement exacte.
 
 \begin{prop}
 Soit $\Kcal=(K_i)_{i\in\N}$ une tour de corps de nombres. Si $\Kcal$ est telle que 
 \[\lim_{i\to\infty}\frac{\log[K_i:K_{i-1}]}{[K_{i-1}:\Q]}=0,\]
alors $\Kcal$ v\'erifie  le th\'eor\`eme de Brauer--Siegel g\'en\'eralis\'e. 
 \end{prop}
 
 \begin{proof} 
Avec le th\'eor\`eme \ref{theo:principal}, il suffit   de v\'erifier que $\Kcal$ est de complexit\'e galoisienne mod\'er\'ee. Pour tout $\varepsilon>0$, il existe par hypoth\`ese un entier $J\geq 1$ tel que 
\[ \frac{\log[K_j:K_{j-1}]}{[K_{j-1}:\Q]}\leq \varepsilon\, \frac{\log 3}{4}\] 
pour tout $j\geq J$. Notons $M_J := \max_{1\leq j \leq J} \log\comp_{K_{j}/K_{j-1}}$. Fixons un entier $J'\geq J$ tel que $M_J \leq \varepsilon\, g(K_i)$ pour tout $i\geq J'$. Pour tous entiers $J< j\leq i$, la borne triviale $\comp_{K_j/K_{j-1}}\leq [K_j:K_{j-1}]!$,  l'in\'egalit\'e   $g(K_i)\geq  g(K_j)$ et la borne de Minkowski \eqref{ineq:minkowski} donnent alors la majoration :  
\begin{align*}
\frac{\log\comp_{K_j/K_{j-1}}}{g(K_i)} 
&\leq \frac{\log\big([K_j:K_{j-1}]!\big)}{g(K_i)} 
\leq \frac{[K_j:\Q]}{g(K_j)}\, \frac{\log[K_j:K_{j-1}]}{[K_{j-1}:\Q]} 
\leq \frac{4}{\log 3}\, \frac{\log[K_j:K_{j-1}]}{[K_{j-1}:\Q]} \leq \varepsilon.
\end{align*}
Par suite, nous avons 
\[\frac{\log\comp_{K_i/\Q}}{g(K_i)}\leq \max_{1\leq j \leq i} \left\{\frac{\log\comp_{K_j/K_{j-1}}}{g(K_i)}\right\} = \max\left\{ \frac{M_J}{g(K_i)}, \max_{J< j \leq i}\left\{ \frac{\log\comp_{K_j/K_{j-1}}}{g(K_i)} \right\}\right\} \leq \varepsilon\]
pour tout entier $i\geq J'$. La tour $\Kcal$ est donc de complexit\'e galoisienne mod\'er\'ee. 
\end{proof}

\subsection{Extensions infinies}
Nous dirons, suivant en cela \cite{TV}, qu'une extension alg\'ebrique $\mathcal L$ de~$\Q$ {\it v\'erifie le th\'eor\`eme de Brauer--Siegel g\'en\'eralis\'e} s'il existe une tour de corps de nombres dont la r\'eunion est $\mathcal L$ qui v\'erifie le th\'eor\`eme de Brauer--Siegel g\'en\'eralis\'e (c'est alors le cas pour toute tour de corps de nombres dont la r\'eunion est $\mathcal L$  d'apr\`es les r\'esultats de \cite[\S 7]{TV}).
\begin{prop}\label{prop:sur.ext}
Soit $(K_j)_{j\in\N}$ une famille de corps de nombres de complexit\'e galoisienne mod\'er\'ee.
Toute extension alg\'ebrique  $\mathcal{L}$ de $\Q$ contenant 
$\bigcup_{j\in\N} K_j$ satisfait le th\'eor\`eme de Brauer--Siegel g\'en\'eralis\'e.
\end{prop}

\begin{proof}
Fixons une premi\`ere tour $(L_i)_{i\in\N}$ de corps de nombres dont la r\'eunion est $\mathcal L$ et construisons par r\'ecurrence une seconde tour $(M_i)_{i\in\N}$ de r\'eunion $\mathcal L$ et dont la complexit\'e galoisienne est mod\'er\'ee. Nous posons $M_0= \Q$ et, \`a $i\geq 1$ donn\'e, nous d\'efinissions $M_i$ \`a partir de $M_{i-1}$ comme suit. Pour tout $j\geq 1$,  posons $M_{i,j} := M_{i-1}\cdot L_i \cdot K_j$. Les inclusions $\Q\subset K_j \subset M_{i,j}$ impliquent  
\[ \comp_{M_{i,j}/\Q} \leq \max\{\comp_{K_j/\Q}, \comp_{M_{i,j}/K_j}\} \leq \max\{\comp_{K_j/\Q}, [M_{i,j}:K_j] !\} \leq \max\{\comp_{K_j/\Q}, [M_{i-1}\cdot L_i :\Q] !\}.\]
Ainsi,  le quotient $(\log\comp_{M_{i, j}/\Q})/g(K_j)$ tend vers $0$ lorsque $j$ tend vers $+\infty$. Soient $j_i$ l'entier minimal tel que $(\log\comp_{M_{i, j_i}/\Q})/g(K_{j_i})\leq 2^{-i}$ et $M_i := M_{i,j_i}$. Puisque $g(K_{j_i})\leq g(M_i)$, le quotient $(\log \comp_{M_i/\Q})/g(M_i)$ tend vers $0$ lorsque $i$ tend vers l'infini. Par suite la tour $(M_i)_{i\in\N}$, dont la r\'eunion est $\mathcal L$, est  de complexit\'e galoisienne mod\'er\'ee. Ceci conduit \`a la conclusion souhait\'ee.
\end{proof}

Le r\'esultat ci-dessus permet d'obtenir  de nombreux corollaires. Par exemple, une extension \linebreak alg\'ebrique de $\Q$ contenant une extension galoisienne par pas $\mathcal K /\Q$ de degr\'e infini v\'erifie le th\'eor\`eme de Brauer--Siegel g\'en\'eralis\'e. \'Enon\c{c}ons \'egalement le cas particulier suivant:
\begin{coro}
Soit $\mathcal{K}$ une extension alg\'ebrique de $\Q$. 
On suppose que $\mathcal K$ contient une infinit\'e de corps quadratiques.
Alors $\mathcal{K}$ satisfait le th\'eor\`eme de Brauer--Siegel g\'en\'eralis\'e.
\end{coro}

\begin{proof}
Appliquer la proposition \ref{prop:sur.ext} \`a la famille des corps quadratiques. 
\end{proof}

Ce corollaire est \'etonnant en ce que les preuves d'autres cas particuliers de la conjecture de Tsfasman--Vl\u{a}du\c{t}  se basent usuellement sur la finitude du nombre de corps quadratiques contenus dans les corps de nombres de la famille consid\'er\'ee (voir par exemple  le lemme 7.3 et le th\'eorème 7.3 de \cite{TV}).

\subsection{Exemples explicites}\label{ss:exemples.explicites} 
 Voici quelques illustrations d'application  du corollaire~\ref{coro:pro.plusieurs}, tir\'ees des articles \cite{Mpmj}, \cite{HM1}, \cite{HM2} et \cite{HMR21}.

\begin{exple} Dans \cite[\S3]{Mpmj}, Maire  construit explicitement des corps de nombres $M/\Q$ dont la $2$-tour de Hilbert $\mathcal{M}_{2}$ est de degr\'e infini sur $M$. Les corps globaux infinis $\mathcal{M}_{2}$ de \cite[\S3]{Mpmj} v\'erifient le th\'eor\`eme de Brauer--Siegel g\'en\'eralis\'e, par le corollaire~\ref{coro:pro.plusieurs}.

Maire utilise ces corps infinis pour produire des exemples de corps de nombres $K$ dont  la tour de Hilbert~$\mathcal{H}_K/K$ est finie mais dont l'extension galoisienne maximale non ramifi\'ee $\mathcal{K}_{\infty}/K$ est infinie (voir son th\'eor\`eme $3.1)$. Comme l'extension $\mathcal{K}_\infty/K$ est galoisienne, nos r\'esultats impliquent que $\mathcal{K}_{\infty}$ v\'erifie le th\'eor\`eme de Brauer--Siegel g\'en\'eralis\'e. On ne peut malheureusement pas prouver la m\^eme conclusion pour les familles form\'ees de sous-extensions arbitraires de $\mathcal{K}_\infty/K$.  
\end{exple}

\begin{exple}
Dans ce m\^eme article, Maire construit (voir \cite[th\'eor\`eme	 5.1]{Mpmj}) \'egalement une suite infinie $(F_i)_{i\in\N}$  de corps quadratiques sur $\Q$ tels que, pour tout $i$, la tour de Hilbert $\mathcal{H}_{F_i}$ est une extension finie de~$F_i$ bien que $F_i$ admette une extension non ramifi\'ee infinie $\mathcal{K}_i$ de degr\'e surnaturel~$2^\infty$. Chacun de ces corps globaux infinis $\mathcal{K}_i$ v\'erifie le th\'eor\`eme de Brauer--Siegel g\'en\'eralis\'e.
\end{exple}

\begin{exple}  Dans leurs exemples A1, B1 et C1 de \cite[\S3.2]{HM1} ainsi que ceux de \cite[\S3]{HM2}, Hajir et Maire construisent explicitement des sp\'ecimens de corps de nombres $K$ et d'ensembles finis~$T$ de places de $K$ tels que l'extension $K_T(2)/K$ est infinie. Dans chacun de ces exemples,  notre corollaire~\ref{coro:pro.plusieurs} montre que l'extension $K_{T}(2)$ v\'erifie le th\'eor\`eme de Brauer--Siegel g\'en\'eralis\'e.
\end{exple}

\begin{exple} 
Les exemples de \cite[\S3]{HM2} ont plus tard \'et\'e repris par Zykin dans \cite[th\'eor\`eme~4]{Zykin}. Il obtient  des tours asymptotiquement bonnes $\mathcal{L}$ pour lesquelles $\beta(\mathcal{L})$ est, conditionnellement \`a GRH, exceptionnellement petit. Sp\'ecifiquement, Zykin montre que  \[0,56497 \leq \beta(\mathcal{L}) \leq 0,59749.\] 
Dans cet encadrement, la minoration  de $\beta(\mathcal{L})$ reste valable inconditionnellement, seule la majoration n\'ecessite GRH. Avec la m\^eme m\'ethode de programmation lin\'eaire, on pourrait toutefois obtenir  une majoration moins pr\'ecise, mais inconditionnelle, de $\beta(\mathcal{L})$ en utilisant la version inconditionnelle (plus faible) de l'in\'egalit\'e fondamentale de Tsfasman--Vl\u{a}du\c{t}.

Zykin consid\`ere un corps totalement complexe $K$ de degr\'e $12$ qui n'est ni galoisien par pas  ni \`a cl\^oture galoisienne  r\'esoluble sur $\Q$ ($\mathfrak{S}_{6}$ est un quotient de son groupe de Galois) mais admet une $\{\mathfrak{p}\}$-$2$-tour  de corps de classes $\mathcal{L}=K_{\{\mathfrak{p}\}}(2)$ infinie  mod\'er\'ement ramif\'ee, o\`u $\mathfrak{p}$ est un id\'eal premier de norme $9$ (voir \cite{HM2}). 

Avant le pr\'esent travail, supposer GRH \'etait la seule fa\c{c}on d'assurer que $\mathcal L$ v\'erifie la conjecture de Tsfasman--Vl\u{a}du\c{t}. Avec notre corollaire~\ref{coro:pro.plusieurs}, on sait \`a pr\'esent que la tour $\mathcal{L}$  v\'erifie le th\'eor\`eme de Brauer--Siegel g\'en\'eralis\'e. Plus pr\'ecis\'ement, si l'on \'ecrit $\mathcal{L}=\bigcup_i L_i$, la limite
\[ \mathcal{BS}(\mathcal{L}) := \lim_{i\to\infty}\frac{\log\big(h(L_i)\ R(L_i)\big)}{g(L_i)}\]
existe et l'on a \emph{inconditionnellement} $ \BS(\mathcal{L}) = \beta(\mathcal L) \geq 0,56497$.   
\end{exple}

\begin{exple}
Il suit aussi du corollaire~\ref{coro:pro.plusieurs} que les exemples obtenus par \og{}coupe dans \linebreak $\Gal(K_S(p)/K)$\fg{}  dans le tout r\'ecent travail \cite{HMR21}  v\'erifient le th\'eor\`eme de Brauer--Siegel g\'en\'eralis\'e. Cela sugg\`ere de belles perspectives : apr\`es avoir donn\'e des versions effectives de leurs r\'esultats, on pourra, avec l'approche de \cite{Zykin}, exhiber des tours de corps de nombres~$\mathcal{L}$ dont les ratios de Brauer--Siegel $\BS(\mathcal{L})$ sont remarquablement petits.
\end{exple}

\bigskip
\noindent\hfill\rule{7cm}{0.5pt}\hfill\phantom{.}
\medskip

\noindent \textbf{\large Remerciements --} 
Les auteurs remercient St\'ephane Louboutin et Christian Maire  pour leurs remarques sur une version pr\'ec\'edente de ce travail, celles-ci ont permis d'am\'eliorer nos résultats et leur pr\'esentation. 
Le premier auteur a \'et\'e en partie financ\'e par la Swiss National Science Foundation  (au travers de la bourse SNSF \#170565 attribu\'ee \`a Pierre Le Boudec), le second   par le projet {\it GA CROCOCO} de la r\'egion Bourgogne Franche-Comt\'e.
Les deux premiers auteurs ont \'et\'e partiellement financ\'es par le projet ANR-17-CE40-0012 {\it FLAIR}.
\medskip

\newcommand{\mapolicebackref}[1]{\hspace*{-2pt}{\textcolor{teal}{\footnotesize $\uparrow$ #1}}}
\renewcommand*{\backref}[1]{\mapolicebackref{#1}}
\hypersetup{linkcolor= teal!80}

\vfill
\noindent\rule{7cm}{0.5pt}

\smallskip
\noindent
{Richard {\sc Griffon}} {(\it \href{richard.griffon@uca.fr}{richard.griffon@uca.fr})}  --
{\sc Laboratoire de Math\'ematiques Blaise Pascal, Universit\'e Clermont Auvergne}. 
Campus Universitaire des C\'ezeaux,
3 place Vasarely,
TSA 60026 CS 60026,
63178 Aubi\`ere Cedex (France). 

\smallskip
\noindent 
{Philippe {\sc Lebacque}} {(\it \href{philippe.lebacque@univ-fcomte.fr}{philippe.lebacque@univ-fcomte.fr})}  --
{\sc Laboratoire de Math\'ematiques de Besan\c con, Universit\'e Bourgogne Franche-Comt\'e}. 
UFR Sciences et Techniques,
16 route de Gray, 25030 Besan\c con (France).

\smallskip
\noindent 
{Ga\"el {\sc R\'emond}} {(\it \href{Gael.Remond@univ-grenoble-alpes.fr}{Gael.Remond@univ-grenoble-alpes.fr})}  --
{\sc Institut Fourier, Universit\'e Grenoble Alpes}.
100 rue des math\'ematiques,
CS 40700,
38058 Grenoble Cedex 9
 (France).

\end{document}